\documentclass[11pt]{amsart}
\usepackage{amsmath,amssymb,amsthm,latexsym,cite,cancel}
\usepackage[small]{caption}
\usepackage{graphicx,wasysym,overpic,tikz,color}
\usepackage{subfigure,color}
\usepackage{cite}
\usepackage[colorlinks=true,urlcolor=blue,
citecolor=red,linkcolor=blue,linktocpage,pdfpagelabels,
bookmarksnumbered,bookmarksopen]{hyperref}
\usepackage[italian,english]{babel}
\usepackage{units}
\usepackage{enumitem}
\usepackage[left=2.6cm,right=2.6cm,top=2.71cm,bottom=2.71cm]{geometry}
\usepackage[hyperpageref]{backref}

\usepackage[colorinlistoftodos]{todonotes}
\newcommand{\R}{\mathbb{R}}
\newcommand{\RD}{{\mathbb{R}^2}}
\newcommand{\J}{J_{e,T}}
\renewcommand{\c}{c_{e,T}}
\renewcommand{\H}{H^1(\RD)}
\newcommand{\Hr}{H^1_r(\RD)}
\newcommand{\n }{\nabla }
\renewcommand{\a }{\alpha }

\newcommand{\de}{\partial}
\newcommand{\weakto}{\rightharpoonup}
\newcommand{\dis}{\displaystyle}

\numberwithin{equation}{section}
\newtheorem{theorem}{Theorem}[section]
\newtheorem{proposition}[theorem]{Proposition}
\newtheorem{lemma}[theorem]{Lemma}

\theoremstyle{definition}

\title[Standing waves for a Schr\"odinger-Chern-Simons-Higgs system]{Standing waves for a Schr\"odinger-Chern-Simons-Higgs system}

\author[P.\ d'Avenia]{Pietro d'Avenia}

\address{Dipartimento di Meccanica, Matematica e Management
\newline\indent 
Politecnico di Bari
\newline\indent
Via Orabona 4,  I-70125  Bari, Italy}
\email{\href{mailto:pietro.davenia@poliba.it}{pietro.davenia@poliba.it}}

\author[A. Pomponio]{Alessio Pomponio}

\address{Dipartimento di Meccanica, Matematica e Management
\newline\indent 
Politecnico di Bari
\newline\indent
Via Orabona 4,  I-70125  Bari, Italy}
\email{\href{mailto:alessio.pomponio@poliba.it}{alessio.pomponio@poliba.it}}

\thanks{}


\thanks{The authors are partially supported by FRA2016 founds of Politecnico di Bari. 
The first author is also supported by the FFABR founds.}

\begin{document}

\maketitle

\begin{abstract}
We consider a system arising from a nonrelativistic Chern-Simon-Higgs model, in which a charged field is coupled with a gauge field. We prove an existence result for small coupling constants.
\end{abstract}
	
\section{Introduction}

In this paper we deal with the following system
\begin{equation}\label{sistema}
\begin{cases}
\displaystyle
-\frac{1}{2m}\Delta u 
+\omega u
+ \frac{e^4}{2m\kappa^2}\frac{h_{u}^{2}(|x|)}{|x|^{2}} u
+ \frac{e^{4}}{ m \kappa^{2}} u \int_{|x|}^{\infty}\frac{u^{2}(s)}{s} h_{u}(s) ds\\
\qquad\qquad\qquad\qquad\qquad\qquad\qquad\qquad
\displaystyle +\left(1+\frac{\kappa q}{2 m}\right)N u 
+ \frac{q}{4m^2} u^3=0, &\quad\hbox{ in }\RD,
\\[5mm]
\dis - \Delta N + \kappa^2 q^2 N + q\left(1+\frac{\kappa q}{2m}\right) u^2=0,&\quad\hbox{ in }\RD,
\end{cases}
\end{equation}
where $u,N:\R^2 \to \R$ are radially symmetric, 
\[
h_{u}(|x|)=
\int_{0}^{|x|}s u^{2}(s)ds,
\] 
and $m,\omega,e, \kappa,q$ are positive constants with a suitable physical meaning.

This system appears in $(2+1)$-dimensional Abelian Higgs models (see Section \ref{deduction} for some details) in which a charged nonrelativistic matter field interacts with a massive gauge field containing only a Chern-Simons term in the gauge action (see \cite{DT}).

In the last years, this type of systems received a great interest both from a mathematical and physical point of view
due, in particular, to the important role of the soliton-like solutions.

From a physical point of view, it provides a second-quantized description of a fixed number of nonrelativistic particles moving in $\delta$-function potentials 
and interacting with massive relativistic photons and neutral scalars.
Such models appear, for instance, as field-theoretic models for anyons, as effective theories for anyon superconductivity, as realistic models of fractional statistics (see \cite{JLW,JP1,JP2,JW,HKP,PK} and references therein).

From a mathematical point of view, nonlocal systems of this type have been object of intensive study in last years in two almost opposite directions.\\
In the first one, we can include the nonlocal Schr\"odinger-Chern-Simons equation, when there is no interaction with the neutral scalar filed $N$. We refer, for instance, to \cite{byeon,BHS2,cunha,DPW,huh2,JPR,P,AD,AD2,WT,Y}. In particular, in  \cite{byeon,AD},
the peculiar nature of the nonlocal term is deeply analysed.
\\
In \cite{HHJ2014}, on the contrary, the authors deal with the nonrelativistic Chern-Simon-Higgs model, in which a charged field is coupled with a gauge field, assuming it identically $0$. Therefore, on one hand, their system is simpler than the nonlocal Schr\"odinger-Chern-Simons equation due to the lack of the nonlocal Chern-Simons term; on the other hand they have to deal with a more involved two-variables energy functional. 
In order to do this, they first solve the second equation, for any $u$, fixed finding $N_u$, the unique solution, then they can treat a one-variable nonlocal energy functional and a sharp study of the properties of the map $u\mapsto N_u$ is required. 

Our aim is, therefore, to join these two almost opposite research directions and to treat \eqref{sistema} in its whole complexity. 
The novelty with respect to the {\it classical} Schr\"odinger-Chern-Simons systems
is the presence of a specific form of the Higgs potential, which modifies the geometry of the energy functional, 
and of a neutral scalar field $N$, which implies a further {\it nontrivial nonlocality} of the problem.
Indeed, as in \cite{HHJ2014}, we have to deal with a problematic two-variables energy functional and, as first step, we solve the second equation for any $u$ fixed finding $N_u$, the unique solution. 
Then we have to treat a one-variable energy functional with two different nonlocal terms, the Chern-Simons one and that one related to $N_u$, which are in strong competition each other. This creates several non-trivial obstacles to the application of classical variational techniques: the geometrical aspects of the functional and the boundedness of its Palais-Smale sequences are not clear and standard at all. In order to overcome such difficulties,
following \cite{JC}, we introduce a truncated functional with more usual geometrical properties and  for which we are able to find a critical point which is, actually, a critical point of the our un-truncated functional requiring that the coupling constant $e$ is  sufficiently small. Clearly looking for a nontrivial solution for \eqref{sistema} for {\em any fixed} $e$ is still an open problem and a stimulating challenge but, up to our knowledge, this is the first result for this physical model in presence of a nontrivial electromagnetic field. Another interesting open problem is the study of the Chern-Simons limit, namely the behavior of the solutions as the gauge coupling constant $q$ goes to infinity.

Our main result is the following
\begin{theorem}\label{main}
There exists a nontrivial solution of system \eqref{sistema} for any coupling constant $e$ sufficiently small.
\end{theorem}

This paper is organized as follows. 
In Section \ref{deduction}, we introduce the derivation of \eqref{sistema}, explaining the physical motivation of such model 
and the role of all the physical constants. 
In Section \ref{se:proof} we present the variational formulation of the problem and prove our main result.

\medskip
We conclude with some notations. In the following we denote by $\|\cdot\|$ and by $\|\cdot \|_p$ respectively the usual norms in $\H$ and in $L^p(\RD)$, with $1\le p\le +\infty$. All the integrals, if not specified differently, are evaluated on $\RD$ with respect to the standard Lebesgue measure. Finally, $C,C_i$ are fixed independent constants  which may vary from line to line.

\section{Deduction of the system}\label{deduction}

As in \cite{DT}, our starting point is the $(2+1)$-dimensional relativistic Lagrangian density 
\begin{align*}
\mathcal{L}^{\rm R}(\phi,N,A_0,A_1,A_2)
&=
\underbrace{\overline{D_\mu\phi} D^\mu\phi
+\frac{1}{2q} \partial_\mu N \partial^\mu N
-\frac{1}{c^2} |\phi|^2 \left(N+\frac{1}{\kappa c}v^2\right)^2
-\frac{q}{2c^2} (|\phi|^2+\kappa c N)^2}_{\rm matter}\\
&\qquad
-\underbrace{\frac{1}{4q} F_{\mu\nu}F^{\mu\nu}}_{\rm Maxwell}
+\underbrace{\frac{\kappa}{4}\varepsilon^{\mu\nu\alpha}F_{\mu\nu}A_\alpha}_{\rm Chern-Simons}
\end{align*}
in the Minkowski space $\R\times\R^2$ with metric $(1,-1,-1)$, where we use Greek letters to denote space-time indices $0,1,2$ and Latin letters for spatial indices $1,2$,
$\phi:\R\times\R^2\to \mathbb{C}$ is a charged scalar field, $N:\R\times\R^2 \to \R$ is a neutral scalar field, $A_\mu:\R\times\R^2 \to \R$ are the components of the gauge potential, $F_{\mu \nu}:=\partial_\mu A_\nu - \partial_\nu A_\mu$, $q>0$ is the gauge coupling constant, $\kappa$ is the Chern-Simons coupling constant, $c$ is the velocity of the light, $v$ is a real constant, $\varepsilon^{\mu\nu\alpha}$ is the Levi-Civita tensor and 
\[
D_\mu = \partial_\mu + \frac{ie}{c} A_\mu,
\]
being $e>0$ the so called coupling constant.

As it is well described in \cite{DT}, $\mathcal{L}^{\rm R}$ represents an interpolation between the pure Chern-Simons interaction model (see \cite{HKP,JW}) and the Abelian Higgs model with pure Maxwell interaction (see \cite{B}).

Indeed, the Lagrangian density $\mathcal{L}^{\rm R}$ is the sum of three types of terms.
The first one is due to the matter field $\phi$ which interacts with the electromagnetic field through the covariant derivatives $D_\mu$ and contains a neutral scalar field $N$ (with mass equal to the gauge field mass).
The last two terms are, respectively, the classical Maxwell Lagrangian density and the additional Chern-Simons term.\\
Even if, classically, the relativistic Abelian Higgs model presents both Chern-Simons and Maxwell terms for the gauge field in the Lagrangian density, recently, physical models in which the Chern-Simons term represents the entire gauge field action have been studied (see e.g. \cite{HKP,JLW,JP1,JP2,JW}). Indeed, for instance, at large distances and low energies the lower derivatives of the Chern-Simons term dominate the higher derivative appearing in the Maxwell term; hence
this last term becomes negligible.\\
%
%
%
%
%
%
%
%
%
Taking
\[
\phi (t,x) = \frac{1}{\sqrt{2m}} (e^{-imc^2 t}\psi(t,x)+e^{imc^2 t}\tilde{\psi}(t,x)) 
\]
and repeating the same arguments of Dunne and Trugenberger \cite{DT}, namely, dropping  all terms which oscillate as $c\to +\infty$, keeping only dominant inverse powers of $c$, considering the zero-antiparticle sector ($\tilde{\psi}=0$) and removing the Maxwell term, we arrive to 
\begin{align*}
\mathcal{L}(\psi,N,A_0,A_1,A_2)
&=
i \bar{\psi} D_t \psi
-\frac{1}{2m}  |{\bf D}\psi|^2
+\frac{1}{2q}\partial_\alpha N \partial^\alpha N 
- \frac{1}{2} \kappa^2 q N^2
-\frac{q}{8m^2c^2} |\psi|^4\\
&\qquad
- \left(1+\frac{\kappa q}{2mc}\right) |\psi|^2 N
+\frac{\kappa}{4} \varepsilon^{\mu\alpha\beta}A_\mu F_{\alpha\beta},
\end{align*}
where $m=\frac{v^2}{\kappa c^3}$ is the mass of the electric potential $A^0$,  $D_t=\partial_t+ieA^0$, ${\bf D}=\nabla - ie {\bf A}$, $(A^0,{\bf A})=(A^0,A^1,A^2)=(A_0,-A_1,-A_2)$ due to the choice of the metric.

Making the variations with respect to all the components and taking, as usual and for simplicity, $c=1$, we get
\[
iD_t\psi + \frac{1}{2m} {\bf D}^2 \psi - \frac{q}{4m^2} |\psi|^2 \psi - \left(1+\frac{\kappa q}{2m}\right) \psi N = 0,
\]
\[
\partial_{tt} N - \Delta N + \kappa^2 q^2 N + q\left(1+\frac{\kappa q}{2m}\right) |\psi|^2=0,
\]
\[
\kappa(\partial_2 A^1 - \partial_1 A^2)= e |\psi|^2,
\]
\[
\kappa m (\partial_2 A^0 + \partial_0 A^2) = e\mathfrak{Im}(\bar{\psi} D_1 \psi),
\]
\[
\kappa m (\partial_0 A^1 + \partial_1 A^0) = -e\mathfrak{Im}(\bar{\psi} D_2 \psi),
\]
in $\R\times\R^2$.

If we look for waves $\psi(t,x)=u(t,x){\rm e}^{iS(t,x)}$, then the Lagrangian density $\mathcal{L}$ depends on $(u,S,N,A^0,A^1,A^2)$ and the corresponding variations are 
\[
-\frac{1}{2m}\Delta u + \left[\frac{1}{2m}|\nabla S - e {\bf A}|^2+(\partial_t S + eA^0)+\left(1+\frac{\kappa q}{2 m}\right)N\right]u + \frac{q}{4m^2} u^3=0,
\]
\[
\partial_t (u^2) + \frac{1}{m} \operatorname{div} ((\nabla S - e {\bf A})u^2)=0,
\]
\[
\partial_{tt} N - \Delta N + \kappa^2 q^2 N + q\left(1+\frac{\kappa q}{2m}\right) u^2=0,
\]
\[
\kappa(\partial_2 A^1 - \partial_1 A^2)= e u^2,
\]
\[
\kappa m (\partial_2 A^0 + \partial_0 A^2)  = e(\partial_1 S - e A^1)u^2,
\]
\[
\kappa m (\partial_0 A^1 + \partial_1 A^0) = - e(\partial_2 S - e A^2)u^2.
\]

In the static case ($A^\mu=A^\mu(x)$ and $N=N(x)$), if we look for standing waves $\psi(t,x)={\rm e}^{i\omega t} u(x)$, with $\omega>0$, in the Coulomb gauge ($\partial_j A^j=0$), the set of the Euler-Lagrange equations becomes
\begin{eqnarray}
&\displaystyle -\frac{1}{2m}\Delta u + \left[\frac{e^2}{2m}|{\bf A}|^2+(\omega + eA^0)+\left(1+\frac{\kappa q}{2 m}\right)N\right]u + \frac{q}{4m^2} u^3=0,&\label{eq1}\\
&{\bf A}\cdot \nabla (u^2)=0,&\label{eq2}\\
&\displaystyle - \Delta N + \kappa^2 q^2 N + q\left(1+\frac{\kappa q}{2m}\right) u^2=0,&\label{eq3}\\
&\kappa(\partial_2 A^1 - \partial_1 A^2)= e u^2,& \label{eq4}\\
&\kappa m \partial_2 A^0  =  - e^2 A^1 u^2,& \label{eq5}\\
&\kappa m \partial_1 A^0 =  e^2 A^2 u^2,& \label{eq6}
\end{eqnarray}
in $\R\times\R^2$.

Since the problem is invariant by translations, to avoid the related difficulties, we look for radial solutions $u$.\\
Hence, if $u$ is radially symmetric, to have that \eqref{eq2} is always satisfied (up to trivial cases), we take $\bf{A}$ tangential, i.e. 
\[\mathbf A=\frac{e}{\kappa}h_{u}(x)\mathbf t,
\quad
\hbox{where }
\mathbf t=(x^{2}/|x|^{2},-x^{1}/|x|^{2}).
\]
Thus, equation \eqref{eq4} implies
\begin{equation}
\label{symm1}
u^2
=
\frac{x\cdot \nabla h}{|x|^2}
\end{equation}
and the Coulomb gauge can be written as
\begin{equation}
\label{symm2}
\frac{x^2 \partial_1 h - x^1 \partial_2 h}{|x|^2}=0
\end{equation}
for $|x|\neq 0$.\\
Hence, combining \eqref{symm1} and \eqref{symm2} we get that $h_{u}$ has to be a radial function and so we can write
\begin{equation}\label{A}
A^{1}(x)=\frac{e}{\kappa}\frac{x^{2}}{|x|^{2}}h_{u}(|x|),
\quad
A^{2}(x)=-\frac{e}{\kappa}\frac{x^{1}}{|x|^{2}}  h_{u}(|x|).
\end{equation}
Finally, \eqref{eq5} and \eqref{eq6} imply that
\begin{equation}
\label{gradelpot}
\nabla A^{0}=-\frac{e^{3}}{m\kappa^{2} } u^{2}(|x|)h_{u}(|x|)  \mathbf n,
\text{ where }  \mathbf n=(x^{1}/|x|^{2}, x^{2}/|x|^{2}),
\end{equation}
and then the electric potential is radial, i.e. $A^{0}(x)=A^{0}(|x|)$.\\
Thus, by \eqref{symm1} we obtain
\[
h_{u}(|x|)=
\int_{0}^{|x|}s u^{2}(s)ds,
\]
and, considering the electric potential $A^0$ null at infinity, by \eqref{gradelpot} we infer
\begin{equation}
\label{A0}
A^{0}(|x|)=\frac{e^{3}}{ m \kappa^{2}} \int_{|x|}^{\infty}\frac{u^{2}(s)}{s} h_{u}(s) ds.
\end{equation}
Hence, replacing \eqref{A} and \eqref{A0} in \eqref{eq1} and considering \eqref{eq3}, we arrive to \eqref{sistema}.

\section{Proof of Theorem \ref{main}} \label{se:proof}

We will look for solutions of \eqref{sistema} as critical points of the functional $I_e: H^1_r(\mathbb{R}^2)\times H^1_r(\mathbb{R}^2)\to\mathbb{R}$ defined by
\begin{align*}
I_e(u,N)
&=
\frac{1}{4m} \|\nabla u\|_2^2
+\frac{\omega}{2}\|u\|_2^2
+ \frac{e^4}{4m\kappa^2} \int \frac{u^2(|x|) h_u^2(|x|)}{|x|^2} 
+ \frac{1}{2} \left(1+\frac{\kappa q}{2 m}\right)\int N u^2\\
& \qquad
+ \frac{q}{16 m^2} \| u\|_4^4
+\frac{1}{4q}\|\nabla N\|_2^2 
+\frac{\kappa^2 q}{4} \|N\|_2^2,
\end{align*}
where
\[
\Hr:=\{u\in \H: u \hbox{ is radially symmetric}  \}.
\]
Observe that the functional $I_e$ is well defined since, for any $u\in \Hr$, we have that
\begin{equation}
\label{ineqhu}
\int \frac{u^2(|x|)h_u^2(|x|)}{|x|^2} 
\leq
C\|u\|^6.
\end{equation}	

We remark that, even if the functional depends on several parameters, we will consider all of them fixed, except $e$, and for this reason we underline this fact with the subscript ``$e$".
 
The two-variables functional $I_e$ is not easy to treat and so, as done in \cite{HHJ2014}, we transform it into a one-variable functional solving the second equation of \eqref{sistema}, for any fixed $u\in \Hr$. Indeed, for every fixed $u\in H^1(\mathbb{R}^2)$, the second equation of \eqref{sistema} has a unique solution that satisfies the following properties, as proved in \cite[Lemma 3.1, Lemma 3.2, Lemma 3.4]{HHJ2014}. 

\begin{lemma}\label{le:N}
For every $u\in H^1(\mathbb{R}^2)$, there exists a unique $N_u\in H^1(\mathbb{R}^2)$ that solves the second equation of \eqref{sistema}. Moreover
\begin{enumerate}[label=(\roman*),ref=\roman*]
\item $N_u \leq 0$ a.e. in $\mathbb{R}^2$;
\item the map $u\mapsto N_u$ is $C^1$;
\item \label{v31}$\displaystyle \left|\int N_u u^2\right|=-\int N_u u^2 \leq\frac{1}{\kappa^2 q}\left(1+\frac{\kappa q}{2m}\right)\|u\|_4^4$;
\item \label{vi31}$N_{tu}=t^2 N_u$, for any $t\in \R$;
\item \label{radial} if $u\in \Hr$, then also $N_u\in \Hr$.
\end{enumerate}
\end{lemma}

Thus, system \eqref{sistema} can be written as the double nonlocal equation
\[
-\frac{1}{2m}\Delta u 
+\omega u
+ \frac{e^4}{2m\kappa^2}\frac{h_{u}^{2}(|x|)}{|x|^{2}} u
+ \frac{e^{4}}{ m \kappa^{2}} u \int_{|x|}^{\infty}\frac{u^{2}(s)}{s} h_{u}(s) ds
+\left(1+\frac{\kappa q}{2 m}\right)N_u u 
+ \frac{q}{4m^2} u^3=0
\]
in $\R^2$ and,  since, by the second equation of \eqref{sistema},
\[
\|\nabla N_u\|_2^2 +\kappa^2 q^2 \|N_u\|_2^2 = -q\left(1+\frac{\kappa q}{2m}\right) \int N_u u^2,
\]
we consider the reduced functional $J_e:\Hr \to \R$ defined as
\begin{align*}
J_e(u)
&: =
I_e(u,N_u)\\
&=
\frac{1}{4m} \|\nabla u\|_2^2
+\frac{\omega}{2}\|u\|_2^2
+ \frac{e^4}{4m\kappa^2} \int \frac{u^2(|x|)h_u^2(|x|)}{|x|^2} 
+ \frac{1}{4} \left(1+\frac{\kappa q}{2 m}\right)\int N_u u^2 
+ \frac{q}{16 m^2} \| u\|_4^4.
\end{align*}
Observe that $J_e$ is of class $C^1$ and 
\begin{align*}
J_e'(u)=\de_u I_e(u,N_u)+\de_N I_e(u,N_u)N'_u=\de_u I_e(u,N_u).
\end{align*}
Finally, arguing as in \cite{byeon,HHJ2014}, we have the following
\begin{lemma}\label{Lemma3.2}
If $u\in \Hr$ is a critical point of $J_e$ then the pair $(u,N_u)$ is a critical point of $I_e$ and so a solution of \eqref{sistema}.
\end{lemma}

In the functional $J_e$ there is a strong competition between the two nonlocal terms. As a consequence, the geometrical aspects of $J_e$ and the boundedness of its Palais-Smale sequences are not clear and standard at all. In order to overcome such difficulties,
following \cite{JC}, we introduce the cut-off function $\chi\in C^\infty(\R_+,\R)$ satisfying 
\[
\begin{cases}
\chi(s)=1 & \hbox{ for } s\in[0,1],\\
0\leq\chi\leq 1 & \hbox{ for } s\in[1,2],\\
\chi(s)=0 & \hbox{ for } s\in[2,+\infty[,\\
\|\chi'\|_\infty \leq 2,
\end{cases}
\]
and, for any $T>0$, define the following truncated functional $\J:\Hr\to \R$ as
\begin{align*}
J_{e,T}(u)
&:=
\frac{1}{4m} \|\nabla u\|_2^2
+\frac{\omega}{2}\|u\|_2^2
+ \frac{e^4}{4m\kappa^2} K_T (u) \int \frac{u^2(|x|)h_u^2(|x|)}{|x|^2} 
+ \frac{1}{4} \left(1+\frac{\kappa q}{2 m}\right)\int N_u u^2 
+ \frac{q}{16 m^2} \| u\|_4^4,
\end{align*}
where 
\begin{equation}\label{kt}
K_T(u):=\chi\left(\frac{\|u\|^2}{T^2}\right).
\end{equation}
Clearly $J_{e,T}$ is of class $C^1$ and if, for certain $T,e>0$, $\bar u\in \Hr$ is a critical point of $\J$ such that $\|\bar u\|< T$, then $\bar u$ is a critical point also for $J_e$ and so, by Lemma \ref{Lemma3.2}, the pair $(\bar u, N_{\bar u})$ is a solution of \eqref{sistema}.

The truncated functional $\J$ satisfies the geometrical assumptions of the Mountain Pass Theorem for all $e,T>0$. More precisely, we have

\begin{proposition}\label{pr:PM}
Independently by $e,T>0$, the functional $\J$ satisfies the following properties:
\begin{enumerate}[label=(\roman*),ref=\roman*]
\item there exist $\rho>0$ and $\alpha>0$  such that $\J(u)\ge \alpha$, for all $u\in H^1_r(\RD)$ such that $\|u\|=\rho$;
\item \label{ii33}there exists $\bar u\in \Hr$, with $\|\bar u\|>\rho$, such that $\J(\bar u)<0$. 
\end{enumerate}
\end{proposition}

\begin{proof}
For the first statement, it is enough to observe that, by (\ref{v31}) in Lemma \ref{le:N}, we have
\begin{align*}
\J(u)
\geq
\frac{1}{4m} \|\nabla u\|_2^2
+\frac{\omega}{2}\|u\|_2^2
- \frac{1}{4\kappa^2 q}\left(1+\frac{\kappa q}{2m}\right)^2  \|u\|_4^4 
+ \frac{q}{16 m^2} \| u\|_4^4
\geq \alpha >0,
\end{align*}
if $\|u\|$ is small enough.
\\
Let us now prove the second part of the proposition. Observe that, for any $u\in \Hr$, by (\ref{v31}) and (\ref{vi31}) of Lemma \ref{le:N}, we infer that
\begin{align*}
\J(tu)
&=t^2\left(\frac{1}{4m} \|\nabla u\|_2^2
+\frac{\omega}{2}\|u\|_2^2\right)
+t^4\left(\frac{q}{16 m^2} \| u\|_4^4
-\frac{1}{4} \left(1+\frac{\kappa q}{2 m}\right)\int |N_{u}| u^2 \right)
\\
&\quad + \frac{e^4 t^6}{4m\kappa^2}K_T(tu) \int \frac{u^2(|x|)h_u^2(|x|)}{|x|^2}. 
\end{align*}
Thus, recalling the definition of $K_T$, see \eqref{kt}, for $t$ sufficiently large, we need only to find $u_\varepsilon$ such that
\begin{equation}
\label{cond}
\frac{q}{4 m^2} \| u_\varepsilon\|_4^4
- \left(1+\frac{\kappa q}{2 m}\right)\int |N_{u_\varepsilon}| u_\varepsilon^2
<0.
\end{equation}
By \cite[Lemma 3.5]{HHJ2014} we know that
\[
\inf_{\Hr} \left(\| u\|_4^4/\int |N_{u}| u^2\right) = \kappa^2 q \left(1+\frac{\kappa q}{2 m}\right)^{-1}
\]
and so, since
\[
\kappa^2 q \left(1+\frac{\kappa q}{2 m}\right)^{-1}
<
\frac{4m^2}{q}\left(1+\frac{\kappa q}{2 m}\right),
\]
\eqref{cond} is satisfied for a suitable $u_\varepsilon$.
\end{proof}

For any $e,T>0$, let us define $\c$ the Mountain Pass level for the functional $\J$, namely
\[
\c:=\inf_{t\in [0,1]}\sup_{\gamma(t)\in \Gamma_{e,T}} \J(\gamma(t)),
\]
where
\[
\Gamma_{e,T}:=\{\gamma\in C([0,1],\Hr):\gamma(0)=0,\|\gamma(1)\|>\rho,\J(\gamma(1))<0\}.
\]
By Proposition \ref{pr:PM}, $\Gamma_{e,T}$ is not empty and $\c>\a$.
Thus, by  the Ekeland Variational Principle, there exists a Palais-Smale
sequence
$\{(u_{e,T})_n\}_n $ in $\Hr$  for $\J$ at level $\c$, namely such that
\begin{equation*}
\J((u_{e,T})_n)\to \c,\qquad \J'((u_{e,T})_n)\to 0,
\end{equation*}
as $n \to +\infty$. 

The next proposition shows that these Palais-Smale sequences are uniformly bounded for $e$ sufficiently small.

\begin{proposition}\label{pr:T}
There exist $\bar T > 0$ independent on $e$ and $e^*:= e(\bar T) > 0$ such that if $0<e < e^*$, then
\[
\limsup_n \|(u_{e,T})_n\|< \bar T.
\]
\end{proposition}

\begin{proof}
Assume by contradiction that for any $T>0$ there exists $e>0$ such that, denoting, by simplicity, by $\{u_n\}_n:=\{(u_{e,T})_n\}_n $ the corresponding Palais-Smale sequence of $\J$ at level $\c$, we have 
\begin{equation}\label{lassurdo}
\limsup_n \|u_n\|\ge  T.
\end{equation}
Observing that 
\begin{align*}
J'_{e,T}(u_n)[u_n]
&=
\frac{1}{2m}\|\nabla u_n\|_{2}^{2}
+ \omega\|u_n\|_{2}^{2} 
+\frac{3e^4}{2m\kappa^2} K_T (u_n) \int \frac{u_n^2(|x|)h_{u_n}^2(|x|)}{|x|^2}\\
&\qquad
+ \frac{e^4}{2m\kappa^2 T^2} K'_T (u_n) \|u_n\|^2 \int \frac{u_n^2(|x|)h_{u_n}^2(|x|)}{|x|^2} 
+  \left(1+\frac{\kappa q}{2 m}\right)\int N_{u_n} u_n^2 
+ \frac{q}{4 m^2} \| u_n\|_4^4,
\end{align*}
we have 
\begin{align*}
\c+o_n(1)\|u_n\|
&=
J_{e,T}(u_n) - \frac{1}{4} J'_{e,T}(u_n)[u_n]\\
&=
\frac{1}{8m}\|\nabla u_n\|_{2}^{2}
+ \frac{\omega}{4}\|u_n\|_{2}^{2} 
-\frac{e^4}{8m\kappa^2} K_T (u_n) \int \frac{u_n^2(|x|)h_{u_n}^2(|x|)}{|x|^2}\\
&\qquad
- \frac{e^4}{8m\kappa^2 T^2} K'_T (u_n) \|u_n\|^2 \int \frac{u_n^2(|x|)h_{u_n}^2(|x|)}{|x|^2} 
\end{align*}
and so
\begin{equation}\label{checasino}
\begin{split}
C \|u_{n}\|^2 +o_n(1)\|u_n\|
&\le 
c_{e,T}
+\frac{e^4}{8m\kappa^2} K_T (u_{n}) \int \frac{u_n^2(|x|)h_{u_n}^2(|x|)}{|x|^2} 
\\
&\qquad
+\frac{e^4}{8m\kappa^2 T^2} K'_T (u_{n}) \|u_{n}\|^2 \int \frac{u_{n}^2(|x|)h_{u_n}^2(|x|)}{|x|^2}.
\end{split}
\end{equation}
As first step, let us estimate the Mountain Pass level $c_{e,T}$.
Let $\bar u\in \Hr$ be as in (\ref{ii33}) of Proposition \ref{pr:PM}. Using (\ref{v31}) in Lemma \ref{le:N} we have
\begin{align*}
c_{e,T}
&\leq
\max_{t\geq 0} J_{e,T}(t\bar u) \\
&\leq
\max_{t\geq 0}\left[
t^2\left(
\frac{1}{4m} \|\nabla \bar u\|_2^2 + \frac{\omega}{2}\|\bar u\|_2^2
\right)
+t^4\left(
\frac{q}{16 m^2} \| \bar u\|_4^4 - \frac{1}{4} \left(1+\frac{\kappa q}{2 m}\right)\int |N_{\bar u}|\bar  u^2 
\right)
\right]\\
&\qquad
+\frac{e^4}{4m\kappa^2}
\max_{t\geq 0}\left[
t^6 \chi\left(\frac{t^2\|\bar u\|^2}{T^2}\right) \int \frac{\bar u^2(|x|)h_{\bar{u}}^2(|x|)}{|x|^2} 
\right]\\
&:= A_1+ A_2(T).
\end{align*}
If $t\leq \sqrt{2}T/\|\bar u\|$, then
\[
A_2(T)\leq \frac{e^4}{4m\kappa^2} \frac{8T^6}{\|\bar u\|^6} \int \frac{\bar u^2(|x|)h_{\bar u}^2(|x|)}{|x|^2} 
:= C_1 e^4 T^6,
\]
otherwise $A_2(T)=0$.\\
Hence we have 
\begin{equation}\label{stimac}
c_{e,T}\le A_1+  C_1 e^4 T^6.
\end{equation}
Moreover, by \eqref{ineqhu},
\begin{align}\label{stima1}
\frac{e^4}{8m\kappa^2} K_T (u_{n}) \int \frac{u_{n}^2(|x|)h_{u_n}^2(|x|)}{|x|^2} 
\leq
C e^4K_T (u_{n}) \|u_{n}\|^6\leq C e^4 T^6
\end{align}
and
\begin{align}\label{stima2}
\frac{e^4}{8m\kappa^2 T^2} K'_T (u_{n}) \|u_{n}\|^2 \int \frac{u_{n}^2(|x|)h_{u_n}^2(|x|)}{|x|^2} 
\leq C e^4 T^6.
\end{align}
Hence by \eqref{checasino}, \eqref{stimac}, \eqref{stima1} and \eqref{stima2}, we have
\begin{equation*}
C \|u_{n}\|^2 +o_n(1)\|u_n\|\leq A_1 + C_1 e^4 T^6
\end{equation*}
and, for $T$ sufficiently large, by \eqref{lassurdo}, we have 
\begin{equation*}
C \|u_{n}\|^2 +o_n(1)\|u_n\|\geq  CT^2-T.
\end{equation*}
Therefore 
\[
CT^2-T\leq A_1 + C_1 e^4 T^6,
\]
which gives a contradiction if $e = e(T)$ is sufficiently small and for large $T$.
\end{proof}

We can now conclude the proof of our main result.

\begin{proof}[Proof of Theorem \ref{main}]
Let $\bar T $ and $e^*$ as in Proposition \ref{pr:T} and fix $0<e<e^*$. Let $\{u_n\}_n:=\{(u_{e,\bar T})_n\}_n $ the corresponding Palais-Smale sequence of $J_{e,\bar T}$ at level $c_{e,\bar T}$. By Proposition \ref{pr:T}, we deduce that, for $n$ sufficiently large, $J_{e,\bar T}(u_n)=J_{e}(u_n)$ and $\{u_n\}_n$ is bounded in $\Hr$.
Therefore, there exists $u_0\in \Hr$ such that $u_n \weakto u_0$ weakly in $\H$.
\\
Let us show that, actually, $u_n \to u_0$ strongly in $\H$, up to a subsequence.
\\
Indeed, we can decompose $J_e':\Hr\to \Hr^*$ as
\[
J_e'(u)=\mathfrak{L}(u)+\mathfrak{K}(u),
\]
where $\mathfrak{L}$ is a bounded invertible linear operator defined as
\[
\mathfrak{L}(u)[\phi]=\frac 1{2m}\int \n u \cdot \n \phi +\omega \int u \phi
\]
and 
\begin{align*}
\mathfrak{K}(u)[\phi]&= \frac{e^4}{2m\kappa^2} \int \frac{u(|x|)\phi(|x|)h_u^2(|x|)}{|x|^2} 
+ \frac{e^4}{m\kappa^2} \int \frac{u^2(|x|)h_u(|x|)}{|x|^2} 
\left(\int_0^{|x|} s u(s)\phi(s)ds\right)
\\
&\qquad
+ \left(1+\frac{\kappa q}{2 m}\right)\int N_u u \phi
+ \frac{q}{4 m^2} \int u^3\phi.
\end{align*}
As observed in \cite[Proof of Lemma 3.3]{HHJ2014} and in \cite[Lemma 3.2]{byeon}, $\mathfrak{K}$ is a compact operator and so we infer, in a standard way, that $u_n \to u_0$ strongly in $\H$ and therefore $u_0$ is a critical point of $J_e$  at level $c_{e,\bar T}$ and so  it is nontrivial. Hence, by Lemma \ref{Lemma3.2}, we conclude that the pair $(u_0,N_{u_0})$ is a solution of \eqref{sistema}.
\end{proof}

\end{document}